\documentclass[preprint,11pt]{elsarticle}
\usepackage{amsfonts}
\usepackage{}
\usepackage{amsmath,amssymb,amsthm}
\usepackage{graphicx}
\usepackage{color}
\usepackage{soul} 
\usepackage{color, xcolor} 

\makeatletter
\def\ps@pprintTitle{%
  \let\@oddhead\@empty
  \let\@evenhead\@empty
  \def\@oddfoot{\reset@font\hfil\thepage\hfil}
  \let\@evenfoot\@oddfoot}
\makeatother

\newtheorem{defin}{Definition}[section]
\newenvironment{definition}{\begin{defin}\rm}{\end{defin}}
\newtheorem{theorem}[defin]{Theorem}
\newtheorem{lemma}[defin]{Lemma}
\newtheorem{proposition}[defin]{Proposition}
\newtheorem{corollary}[defin]{Corollary}
\newtheorem{qu}{Question}
\newtheorem*{maintheorem}{Main Theorem}
\newtheorem{exa}[defin]{Example}

\newenvironment{question}{\begin{qu}\rm}{\end{qu}}

\begin{document}

\begin{frontmatter}

\title{Quotient homomorphisms of Topological \(MV\)-Algebras and Applications\corref{cor2}}

\author[mymainaddress]{Li-Hong Xie}
\ead{yunli198282@126.com}

\author[mymainaddress1]{Jiang Yang\corref{cor1}} 
\ead{yangjiangdy@126.com}

\cortext[cor1]{Corresponding author: Jiang Yang}
\cortext[cor2]{The project is supported by the Natural Science Foundation of Guangdong Province under Grant (No. 2021A1515010381) and the Innovation Project of Department of Education of Guangdong Province (No. 2022KTSCX145).}
\address[mymainaddress]{School of Mathematics and Computational Science, Wuyi University, Jiangmen, Guangdong, 529000, P.R. China}

\address[mymainaddress1]{ School of Mathematical Sciences, Guangxi Minzu University, Nanning, 530006, P.R. China}

\begin{abstract}
For a topological group, the quotient map modulo a subgroup is open and the quotient
map modulo a compact subgroup is perfect. In this paper we prove and develop the
corresponding compact-ideal theory for topological \(MV\)-algebras. We show that if \(I\) is an ideal of a topological \(MV\)-algebra \(A\), then the natural quotient homomorphism \(q:A\longrightarrow A/I\), where \(A/I\) is endowed with the quotient topology, is a continuous open quotient map and \(A/I\) is again a topological \(MV\)-algebra. If, in addition, \(I\) is compact, then \(q\) is perfect.
 
As applications, we study three-space phenomena in topological \(MV\)-algebras. Under
compact-kernel hypotheses we prove three-space theorems for compactness, local compactness,
\(\sigma\)-compactness, Lindel\"ofness and paracompactness under the separation hypotheses stated
below. We also prove a first-countability three-space theorem for locally convex topological
\(MV\)-algebras. 
 
\end{abstract}

\begin{keyword}
topological \(MV\)-algebra \sep compact ideal \sep quotient map \sep closed map \sep perfect map \sep three-space property \sep first countability

\MSC[2020] 06D35 \sep 54H13 \sep 54C10 \sep 54D30 \sep 54D45
\end{keyword}

\end{frontmatter}

\section{Introduction}
 All topological spaces in this paper are assumed to satisfy no separation axioms unless otherwise stated.

In recent decades, a number of algebraic structures associated with logical systems have been studied by many mathematicians.
 For instance, Y. Imai and K. Is\'{e}ki \cite{Imai} introduced BCK-algebras as an algebraic formulation of Meredith's BCK-implicational calculus.
  To investigate many-valued logic by algebraic means, BL-algebras have been defined by
H\v{a}jek \cite{Haj}. Undoubtedly, MV-algebras, which were introduced by Chang \cite{Chang} in order to show {\L}ukasiewicz logic to be standard complete,
are among the most important structures associated with logical systems, where ``MV'' is short for ``many-value''. Furthermore, a number of algebraic structures
associated with logical systems endowed with a topology have been investigated by several mathematicians \cite{Bor,Hav, Rou,Yang}.
In particular, Hoo \cite{Hoo1997} introduced the notion of a topological MV-algebra, which means an MV-algebra $(A,\oplus,\ast,0)$ with a topology such that the operations
$\oplus$ and $\ast$ are continuous functions. Some fundamental properties were investigated by Hoo.

Quotient constructions are central in the algebraic theory of \(MV\)-algebras.  If \(I\) is an ideal of an \(MV\)-algebra \(A\), then the quotient \(A/I\) is again an \(MV\)-algebra, and the natural map
\[
        q:A\longrightarrow A/I
\]
is an \(MV\)-homomorphism.  When \(A\) is endowed with a topology making it a topological \(MV\)-algebra, it is natural to put the quotient topology on \(A/I\).  A basic question is then how much of the topology of \(A\) can be transferred to, or recovered from, the quotient.

For topological groups, a standard fact is that quotient maps modulo subgroups are open, and
quotient maps modulo compact subgroups are perfect. This fact is one reason compact kernels
are effective in topological group theory. The aim of this paper is to establish the corresponding
quotient theorem for topological \(MV\)-algebras.

\begin{maintheorem}
Let \(A\) be a topological \(MV\)-algebra and let \(I\) be an ideal of \(A\). Let
\[
        q:A\longrightarrow A/I
\]
be the natural quotient homomorphism, where \(A/I\) is endowed with the quotient topology. Then \(q\) is continuous and open,  and \(A/I\) is a topological \(MV\)-algebra. If \(I\) is compact, then q is perfect.
\end{maintheorem}

The compact-ideal theorem also suggests a more systematic viewpoint.  In the theory of topological groups one often studies properties of a group \(G\), a closed normal subgroup \(H\), and the quotient \(G/H\) simultaneously.  A property is called a three-space property if any two of these three objects having the property force the third one to have it. For topological \(MV\)-algebras the analogous triple is
\[
        I,\qquad A,\qquad A/I.
\]
Unlike the group case, an ideal is not a translated copy of the zero fibre under homeomorphisms; nevertheless compact ideals provide enough control over quotient fibres to recover a substantial part of the group-theoretic picture. In the compact-ideal case the quotient map is perfect, and therefore many classical properties governed by perfect maps become compact-kernel three-space properties for topological \(MV\)-algebras.  One of the main additional results of this paper is that first countability also has such a three-space form in local convex topological \(MV\)-algebra.

The paper is organized as follows. Section \ref{sec:preliminaries} recalls the needed notation for \(MV\)-algebras and topological \(MV\)-algebras.  Section \ref{sec:quotient-open} proves that natural quotient homomorphisms are open and that quotient \(MV\)-algebras with the quotient topology are topological \(MV\)-algebras.  Section \ref{sec:compact-ideal-closed} proves the compact ideal perfect mapping theorem.  Section \ref{app}  gives applications to three-space properties for topological \(MV\)-algebras. We prove compact-kernel three-space theorems: compactness, local compactness, \(\sigma\)-compactness, Lindel\"ofness and Hausdorff paracompactness are equivalent for \(A\) and \(A/I\). We also prove a first-countability three-space theorem: let \(A\) be a local convex topological \(MV\)-algebra and \(I\) be an ideal of \(A\); then \(A\) is first countable if and only if both \(I\) and \(A/I\) are first countable.

\section{Definitions and Preliminaries}\label{sec:preliminaries}
In this section, we will provide basic terminologies and notations of $MV$-algebras
which are necessary for the understanding of subsequent results.

\begin{definition}\label{Def:1}\cite{Chang,CD,Mun}
An $MV$-algebra is a structure $(A,\oplus,{}^\ast, 0)$ such that $(A,\oplus, 0)$ is a commutative monoid, $x^{\ast\ast}=x$, $x\oplus 0^\ast=0^\ast$ and $(x^\ast\oplus y)^\ast\oplus y=(y^\ast\oplus x)^\ast\oplus x$ for all $x,y\in A$.
\end{definition}

In what follows, unless a singularity arises, we will write $A$ for the MV-algebra $(A,\oplus,\ast, 0)$.
In an $MV$-algebra $A$ for any $x, y \in A$ we put:

\[
        1:=0^\ast,
        \qquad
        x\odot y:=(x^\ast\oplus y^\ast)^\ast,
        \qquad
        x\ominus y:=x\odot y^\ast.
\]
Every MV-algebra $A$ carries a partial order $\le $, called its natural order, defined by
\[
        x\leq y\quad\Longleftrightarrow\quad x^\ast\oplus y=1.
\]
With respect to this order,  $A$ becomes a bounded distributive lattice in which, for all \(A\),
\[
        x\vee y=y\oplus(x\ominus y),
        \qquad
        x\wedge y=x\odot(x^\ast\oplus y).
\]

\begin{exa}
Let $[0,1] = \{x \in \mathbb{R} \mid 0 \leq x \leq 1\}$ be equipped with the truncated addition
$x \oplus y := \min\{1, x + y\}$, negation $x^{\ast} := 1 - x$, and identity $0$.
Then $([0,1], \oplus, \ast, 0)$ is an MV-algebra, known as the standard MV-algebra,
which we denote simply by $[0,1]$.
\end{exa}

\begin{lemma}\label{Lem:1}
Let $A$ be an MV-algebra. For any $x,y,z\in A$, the following statements hold:

\begin{enumerate}
 \item[(1)] $ x =(x \wedge y) \oplus(x\ominus y) $ \cite{CD};
 \item[(2)] $ x\ominus(x\ominus y)=x\wedge y$ \cite{CD};
 \item[(3)] $ x\ominus z\leq (x\ominus y)\oplus(y\ominus z);$
 \item [(4)] $(x\oplus y)\ominus x\leq y $ \cite[Lemma 2.4]{LuanZhaoYang2020};
\item [(5)] $x \leq y \Leftrightarrow  y^\ast \leq x^\ast$ \cite[Theorem 1.4 (vi)]{Chang};
\item [(6)] if $x\leq y$, then $ x \oplus z \leq  y \oplus z, x \odot z \leq  y \odot z$ \cite[Theorem 1.8]{Chang},
 \end{enumerate}
\end{lemma}

\begin{definition}\label{Def:2}\cite{CD}
An ideal of an \(MV\)-algebra \(A\) is a subset \(I\subseteq A\) such that \(0\in I\), \(y\leq x\in I\) implies \(y\in I\),
and \(x,y\in I\) implies \(x\oplus y\in I\).
\end{definition}

Let \( A \) and \( B \) be $MV$-algebras. Recall that a mapping \( f : A \to B \) is a \(MV\)-\textbf{homomorphism} \cite{CD} if it satisfies the following conditions: for all \( x, y \in A \),

\begin{enumerate}
    \item[(i)] \( f(0) = 0 \);
    \item[(ii)] \( f(x \oplus y) = f(x) \oplus f(y) \);
    \item[(iii)] \( f(x^*) = (f(x))^* \).
\end{enumerate}

If \( f \) is one-to-one, we say that \( f \) is \textbf{injective}. If \( f \) is onto, we say that \( f \) is \textbf{surjective}. By an \(MV\)-\textbf{isomorphism}, we mean a surjective and one-to-one \(MV\)-homomorphism. We write \( A \cong B \) if there exists an isomorphism from \( A \) onto \( B \).

For any MV-algebra \( A \), there exists a distance function \cite[Definition 1.2.4]{CD} defined by
\[
d(x, y) := (x \ominus y) \oplus (y \ominus x), \quad \text{for all } x, y \in A,
\]
which is called the {\it Chang distance}.

Recall from \cite[ Proposition 1.2.6]{CD} that for any ideal \( I \) of an MV-algebra \( A \), there exists a corresponding congruence relation \( \equiv_I \) on \( A \), defined by
\[
        x\equiv_I y
        \quad\Longleftrightarrow\quad
        d(x,y)\in I.
\]
The equivalence class of \( x \in A \) with respect to \( \equiv_I \) is denoted by \( x/I \), and the quotient set \( A/I \) is denoted by \( A/I \). Since \( \equiv_I \) is a congruence relation, the quotient set \( A/I \) inherits an MV-algebra structure from \( A \). The operations are naturally defined by

\[
(x/I)^\ast := x^\ast/I , \quad    x/I \oplus y/I := (x \oplus y)/I  \text{~for all~} x,y\in A,
\]
 The resulting system \( (A/I, \oplus, \ast, 0/I) \) is an MV-algebra, called the {\it quotient
algebra of $A$} by the ideal $I$. Moreover, the correspondence \(x \rightarrow x/I\)
defines a homomorphism \(\pi_I\) from \(A\) onto the quotient algebra \(A/I\),
which is called the {\it natural homomorphism} from \(A\) onto \(A/I\) (see \cite{CD}).

 \begin{definition}\cite{Hoo1997}
A topological \(MV\)-algebra is an \(MV\)-algebra \(A\) equipped with a topology such that
\[
        \oplus:A\times A\longrightarrow A,
        \qquad
        {}^\ast:A\longrightarrow A
\]
are continuous. Then the term operations \(\odot,\ominus,\vee,\wedge\) are also continuous.
\end{definition}

For \(U\subseteq A\) and \(a\in A\), define
\[
        U(a)=\{x\in A:a\ominus x\in U\text{ and }x\ominus a\in U\}.
\]
If \(U\) is a neighborhood of \(0\), then \(U(a)\) is a neighbourhood of \(a\) (see \cite[Lemma 3.3]{GLD}), because it is the inverse image of \(U\times U\) under the continuous map
\(     x\longmapsto (a\ominus x,x\ominus a).
\)

\section{ Open quotient \(MV\)-homomorphisms}\label{sec:quotient-open}
One of the important operations on topological universal algebras is that of taking their quotients. In this section, we will discuss
quotient homomorphisms between topological \(MV\)-algebras.

Let \(f:X\longrightarrow Y\) be a surjective map from  topological space \(X\) onto a set \(Y\). The \textbf{quotient topology} on \(Y \) induced by \( f \) is
\[
\tau_f=\{ U \subseteq Y \mid f^{-1}(U) \text{~is open in ~} X\}.
\]

Let \(f:X\rightarrow Y\) be a surjective map between topological spaces. Then \(f\) is called a \textbf{quotient map} in the topological sense if the set \(O\) is open in \(Y\) if and only if \(f^{-1}(O)\) is open in \(X\) \cite[Proposition 2.3]{Engelking1989}. Clearly, the quotient topology \(\tau_f\) on \(Y\) induced by \( f \) is exactly the original topology on \(Y\).

Let \(I\) be an ideal of the topological \(MV\)-algebra \(A\). Then \(A/I\) naturally becomes an \(MV\)-algebra \cite{CD}.  Let \[\pi:A\rightarrow A/I\] be the natural quotient map in the \(MV\)-algebra sense.
Then it is natural to ask whether the \(MV\)-algebra \(A/I\) endowed with the quotient topology induced by \(\pi\) is a topological \(MV\)-algebra.
Recently, Gan, Luan, Deng and Yang \cite{GLD} studied this question. They obtained the following result:

\begin{theorem}\cite[Theorem 5.6]{GLD}\label{Them:GY}
Let \( (A, \tau) \) and \( (B, \pi) \) be topological MV-algebras, and \( f : A \to B \) be an open continuous surjective homomorphism. Let \( \rho : A \to A/\operatorname{Ker} f \) be the natural homomorphism. Then the following statements hold.
\begin{enumerate}
\item[(i)] \( A/\operatorname{Ker}(f) \) endowed with the topology is a topological \(MV\)-algebra.
\item[(ii)] There exists a unique topological isomorphism \(\bar{f}\) from \((A/\operatorname{Ker}(f), \tau_\rho)\) to \((B, \pi)\) such that \(\bar{f} \circ \rho = f\).
\end{enumerate}
\end{theorem}

Clearly, every continuous open map in topological spaces is a quotient map, but the converse is not true. To prove Theorem \ref{Them:GY},  Gan, Luan, Deng and Yang proved the following result:

\begin{proposition}\cite[Proposition 5.1]{GLD}\label{Pro1}
Let \( (A, \tau) \) be a topological \(MV\)-algebra, \( B \) be an MV-algebra and \( f : A \to B \) be a surjective homomorphism from \( A \) onto \( B \) satisfying the condition: \(f^{-1}(f(V)) \in \tau\) holds for each \(V \in \tau\). Then the following statements hold:
\begin{enumerate}
\item[(i)] \( (B, \tau_f) \) is a topological MV-algebra.
\item[(ii)] \( f \) is an open continuous homomorphism from \( (A, \tau) \) onto \( (B, \tau_f) \).
\item[(iii)] If \( (B, \pi) \) is a topological MV-algebra such that \( f \) is an open continuous homomorphism from \( (A, \tau) \) onto \( (B, \pi) \), then \( \tau_f \subseteq \pi \).
\end{enumerate}
\end{proposition}

 In fact, we find that every quotient \(MV\)-homomorphism is continuous and open in topological \(MV\)-algebras, which implies that the condition `\(f^{-1}(f(V)) \in \tau\) holds for each \(V \in \tau\)' in Proposition \ref{Pro1} can be dropped. Also, we improve Theorem \ref{Them:GY} as follows:

\begin{theorem}\label{lem:quotient-open-general}
Let \(A\) be a topological \(MV\)-algebra and \(B\) is an \(MV\)-algebra with a topology \(\tau\). If
\( q:A\longrightarrow B\)
is a surjective \(MV\)-homomorphism and \(q\) is a quotient map in the topological sense, then \(q\) is open. Furthermore, \(B\) with the topology \(\tau\) is a topological \(MV\)-algebra.
\end{theorem}

\begin{proof}
Let \(G\subseteq A\) be open.  Since \(q\) is a quotient map, it is enough to prove that
\[
        q^{-1}(q(G))
\]
is open in \(A\).

Take \(x\in q^{-1}(q(G))\).  Then there exists \(v\in G\) such that
\[
        q(x)=q(v).
\]
Consider the map \(\Theta_v:A\times A\longrightarrow A \) defined by
\[
        \Theta_v(p,r)=(v\ominus p)\oplus r \quad \text{for each~} (p,r)\in A\times A.
\]
Since $A$ is a topological \(MV\)-algebra, the map \(\Theta_v\) is continuous. Hence, there exists a neighbourhood \(W\) of \(0\) such that,
\[
        \Theta_v(W\times W)\subseteq G,
\]
because \(\Theta_v(0,0)=v\in G\) and \(G\) is open.

We claim that
\[
        W(x)\subseteq q^{-1}(q(G)).
\]
Let \(y\in W(x)\).  Then
\[
        x\ominus y\in W,
        \qquad
        y\ominus x\in W.
\]
Put
\[
        w=\Theta_v(x\ominus y,y\ominus x)=(v\ominus(x\ominus y))\oplus(y\ominus x).
\]
By the choice of \(W\), we have \(w=\Theta_v(x\ominus y,y\ominus x)\in \Theta_v(W,W)\subseteq G\).  We show that \(q(w)=q(y)\).  Since \(q\) is an \(MV\)-homomorphism and \(q(v)=q(x)\),
\begin{align*}
q(w)
&=\big(q(v)\ominus(q(x)\ominus q(y))\big)\oplus(q(y)\ominus q(x))\\
&=\big(q(x)\ominus(q(x)\ominus q(y))\big)\oplus(q(y)\ominus q(x))\\
&=(q(x)\wedge q(y))\oplus(q(y)\ominus q(x))\\
&=q(y),
\end{align*}
where the last equality follows from (1) in Lemma \ref{Lem:1}.  Hence \(q(y)=q(w)\in q(G)\), so \(y\in q^{-1}(q(G))\).  Therefore \(W(x)\subseteq q^{-1}(q(G))\).

Since every point of \(q^{-1}(q(G))\) has a neighbourhood contained in \(q^{-1}(q(G))\), this saturation is open.  Hence \(q(G)\) is open by the quotient property of \(q\).  Thus \(q\) is an open map.

It remains to show that \(B\) with the topology \(\tau\) is a topological \(MV\)-algebra.
Since \(f\) is an open continuous
surjection, \(f\times f:A\times A\to B\times B\) is also an open continuous surjection; in particular, it is a
quotient map.

The operation \(\oplus_{B}\) on \(B\) satisfies
\[
        \oplus_{B}\circ(f\times f)=f\circ\oplus_A.
\]
The right-hand side is continuous and \(f\times f\) is open and quotient. Hence \(\oplus_{B}\) is continuous. Similarly, the involution on \(B\) satisfies
\[
        {}^\ast_{B}\circ f=f\circ{}^\ast_A,
\]
and since \(f\) is open and quotient, \({}^\ast_{B}\) is continuous. Therefore \(B\) is a topological \(MV\)-algebra.
\end{proof}

The following Example shows that the condition `\(A\) is a topological \(MV\)-algebra' in Theorem \ref{lem:quotient-open-general} can not weaken to `\(A\) is a paratopological \(MV\)-algebra'.  An \(MV\)-algebra \(A\)
with a topology is called a {\it paratopological \(MV\)-algebra } if the operator \(\oplus:A\times A\rightarrow A\) is continuous. Clearly, every topological \(MV\)-algebra is a paratopological \(MV\)-algebra.

\begin{exa}
Let \(B_4 = \{0, a, b, 1\}\) be the four-element Boolean algebra, where \(0\) and \(1\) are the bottom and top elements, and \(a,b\) are complementary. The MV-algebra operations are defined by
\[
x \oplus y = x \vee y,\qquad x^* = \neg x.
\]
Define a topology \(\tau\) on \(B_4\) by taking as a basis the sets
\[
\{\{0\},\quad \{1\},\quad \{a,1\},\quad \{b,1\}\}.
\]
Equivalently, the open sets are all unions of these basic sets. Then \((B_4,\tau)\) is a paratopological MV-algebra but not a topological MV-algebra. Take the ideal \(I=\{0,a\}\) and Consider the natural quotient
map \(\pi:B_4\rightarrow B_4/I\). Then \(\pi\) is not open.
\end{exa}

\begin{proof}
We first verify that the operation \(\oplus\) is jointly continuous. The product topology on \(B_4 \times B_4\) has a basis of sets \(U \times V\) with \(U,V \in \tau\). It is enough to check that for each basic open set \(W\) of \(B_4\), the preimage \(\oplus^{-1}(W)\) is open in the product topology.

The open sets of \(\tau\) are exactly
\[
\{\emptyset,\ \{0\},\ \{1\},\ \{a,1\},\ \{b,1\},\ \{0,1\},\ \{0,a,1\},\ \{0,b,1\},\ \{a,b,1\},\ B_4\}
\]

\begin{itemize}
\item For \(W = \{0\}\), we have \(\oplus^{-1}(\{0\}) = \{(0,0)\}\), which is \(\{0\}\times\{0\}\), a basic open set.
\item For \(W = \{1\}\), the pairs with \(x \oplus y = 1\) are those with \(x=1\) or \(y=1\), plus \((a,b)\) and \((b,a)\). Thus
\[
\oplus^{-1}(\{1\}) = \bigl(\{1\}\times B_4\bigr) \cup \bigl(B_4\times\{1\}\bigr) \cup \{(a,b),(b,a)\}.
\]
Every point in this set has a basic open neighbourhood contained in it:
\[
(1,y) \in \{1\}\times B_4,\quad (x,1)\in B_4\times\{1\},\quad (a,b)\in \{a,1\}\times\{b,1\},\quad (b,a)\in \{b,1\}\times\{a,1\}.
\]

Hence the preimage is open.
\item For \(W = \{a,1\}\), the pairs with join \(a\) are \((a,0),(0,a),(a,a)\), and those with join \(1\) are as above. Each point has a suitable neighbourhood:
\[
(a,0)\in \{a,1\}\times\{0\},\quad (0,a)\in \{0\}\times\{a,1\},\quad (a,a)\in \{a,1\}\times\{a,1\},
\]
and the cases with \(1\) or \((a,b),(b,a)\) are handled as before.
\item For \(W = \{b,1\}\), the argument is symmetric (interchange \(a\) and \(b\)).
\end{itemize}

Therefore, \(\oplus\) is jointly continuous.

Now consider the complement map \(x \mapsto x^*\). Take the open set \(\{a,1\}\). Its preimage under \(^*\) is
\[
\{x \in B_4 : x^* \in \{a,1\}\} = \{0, b\},
\]
because \(0^*=1,\ a^*=b,\ b^*=a,\ 1^*=0\).

Clearly, the set \(\{0,b\}\) is not open in \(\tau\), so the complement map is not continuous.
Thus \((B_4,\tau)\) is \emph{not} a topological MV-algebra.

The quotient algebra has two equivalence classes: \(0/I = I\) and \(1/I = \{b,1\}\). The quotient topology is \(\tau_\pi = \{U \subseteq B_4/I \mid \pi^{-1}(U) \in \tau\}\). Since \(I = \pi^{-1}(\{0/I\})\) is not open, \(\{0/I\}\) is not open. However, taking the open set \(U = \{0\} \in \tau\), we have \(\pi(U) = \{0/I\}\), which is not open. Therefore \(\pi\) is not an open mapping.
\end{proof}

\begin{corollary}\label{cor:natural-quotient-open}
Let \(A\) be a topological \(MV\)-algebra and \(I\) an ideal of \(A\). Let
\[ q:A\longrightarrow A/I\]
be the natural homomorphism, where \(A/I\) is endowed with the quotient topology. Then \(q\) is a continuous and open map and \(A/I\) is a topological \(MV\)-algebra.
\end{corollary}

\begin{proof}
The map \(q\) is continuous and quotient by the definition of the quotient topology. It is a
surjective \(MV\)-homomorphism. The result follows from Theorem \ref{lem:quotient-open-general}.
\end{proof}

\begin{corollary}
Let \((A, \tau)\) be a topological \(MV\)-algebra, \(B\) be an \(MV\)-algebra and \(f : A \to B\) be a surjective \(MV\)-homomorphism from \(A\) onto \(B\). Then the quotient space \((B, \tau_f)\) is discrete if and only if
\(\operatorname{Ker}(f) \in \tau\), where \(\operatorname{Ker}(f) = f^{-1}(0) = \{x \in A \mid f(x) = 0\}\).
\end{corollary}

\begin{proof}
Assume that \(\operatorname{Ker}(f) \in \tau\). Then \(\{0\}\) is open in \((B,\tau_f)\) because \(f :(A, \tau) \to (B,\tau_f)\) is a quotient mapping. Furthermore, by Theorem \ref{lem:quotient-open-general}, \((B,\tau_f)\) is a topological \(MV\)-algebra. Then \((B, \tau_f)\) is discrete by the fact that topological \(MV\)-algebra \(C\) is discrete if the set \(\{0\}\) is open in \(C\) \cite[Proposition 3.3]{Hoo1997}.

Assume that \((B, \tau_f)\) is discrete. Then \(\{0\} \in \tau_f\) , and so \(f^{-1}(0)=\operatorname{Ker}(f)\in \tau\).
\end{proof}

A mapping \( f : A \to B \) from one topological MV-algebra to another is a \textbf{topological \(MV\)-isomorphism} if \( f \) is both an \(MV\)-isomorphism and a homeomorphism \cite[Definition 3.15]{Hoo1997}.

\begin{corollary}\label{cor:1}
Let \(f:A\longrightarrow B\) be a surjective continuous \(MV\)-homomorphism between topological \(MV\)-algebras, and let \(\pi:A\rightarrow A/\operatorname{Ker}f\) is the natural \(MV\)-homomorphism, where \(A/\operatorname{Ker}f\) endowed with the natural quotient topology. Then
 there is a unique continuous \(MV\)-isomorphism
  \[q:A/\operatorname{Ker}f\longrightarrow B\] such that \(f = q \circ \pi\).
Furthermore, if \(f\) is also a quotient map in the topological sense, then \(q\) is a topological \(MV\)-isomorphism.
\end{corollary}

\begin{proof}

Define \(q : A/\operatorname{Ker}f \longrightarrow B\) by
\[
q(x/\operatorname{Ker}f) = f(x) \quad \text{for any } x \in A.
\]
This is the usual algebraic first isomorphism theorem \cite[Theorem 1.2.8]{CD}.

Continuity follows because, for every open set \(O\) in \(B\),
\[q^{-1}(O)=\pi(f^{-1}(O)),
\]
and \(f^{-1}(O)\) is open while \(\pi\) is open by Corollary \ref{cor:natural-quotient-open}. If \(f\) is a quotient map, Theorem \ref{lem:quotient-open-general} implies
that \(f\) is open, and the same identity gives openness of \(q\).
\end{proof}


\section{Perfect quotient \(MV\)-homomorphisms}\label{sec:compact-ideal-closed}

We now prove the second main theorem.  We begin with two elementary lemmas.

\begin{lemma}\label{lem:compact-parameter-closed}
Let \(X,Y\) be topological spaces, let \(K\) be compact, and let \( \Phi:X\times K\longrightarrow Y\)
be continuous.  If \(F\subseteq Y\) is closed, then \(  S_F=\{x\in X:\Phi(\{x\}\times K)\cap F\neq\emptyset\}\)
is closed in \(X\).
\end{lemma}

\begin{proof}
Let \(x_0\notin S_F\).  Then
\[
        \Phi(\{x_0\}\times K)\subseteq Y\setminus F.
\]
For each \(k\in K\), the point \(\Phi(x_0,k)\) belongs to the open set \(Y\setminus F\).  By continuity of \(\Phi\), there exist neighbourhoods \(U_k\) of \(x_0\) in \(X\) and \(V_k\) of \(k\) in \(K\) such that
\[
        \Phi(U_k\times V_k)\subseteq Y\setminus F.
\]
The family \(\{V_k:k\in K\}\) covers \(K\).  Since \(K\) is compact, choose \(k_1,\ldots,k_n\) such that
\[
        K\subseteq V_{k_1}\cup\cdots\cup V_{k_n}.
\]
Put
\[
        U=U_{k_1}\cap\cdots\cap U_{k_n}.
\]
Then \(U\) is a neighbourhood of \(x_0\), and
\[
        \Phi(U\times K)\subseteq Y\setminus F.
\]
Hence \(U\cap S_F=\emptyset\).  Thus \(X\setminus S_F\) is open, so \(S_F\) is closed.
\end{proof}

\begin{lemma}\label{Lem:2}
Let $A$ be an MV-algebra. For any $x,y,z\in A$, if \(x\leq y\), then \(z\ominus y\leq z\ominus x\) holds.
\end{lemma}

\begin{proof}
If \(x\leq y\), then \(y^\ast\leq x^\ast\) by (5) in Lemma \ref{Lem:1}. Hence, by  (6) in Lemma \ref{Lem:1} we obtain that \(z\odot y^\ast\leq z\odot x^\ast \), that is, \(z\ominus y\leq z\ominus x\).
\end{proof}

\begin{proposition}\label{lem:class-parametrization}
Let \(A\) be an \(MV\)-algebra and \(I\) an ideal of \(A\).  Then \( x/I=\{(x\ominus i)\oplus j:i,j\in I\}\) holds for every \(x\in A\).
\end{proposition}

\begin{proof}
Let \(y\in x/I\).  Then
\[
        x\ominus y\in I,
        \qquad
        y\ominus x\in I.
\]
Put
\[
        i=x\ominus y,
        \qquad
        j=y\ominus x.
\]
Using (1) and (2) in Lemma \ref{Lem:1}, we obtain
\[
        (x\ominus i)\oplus j
        =\big(x\ominus(x\ominus y)\big)\oplus(y\ominus x)
        =(x\wedge y)\oplus(y\ominus x)
        =y.
\]
Thus
\[
        x/I\subseteq\{(x\ominus i)\oplus j:i,j\in I\}.
\]

Conversely, suppose
\[
        y=(x\ominus i)\oplus j
\]
for some \(i,j\in I\).  Clearly, \(x\ominus i\leq y\) holds, so we obtain by Lemma \ref{Lem:2} and (2) in Lemma \ref{Lem:1}:
\[
        x\ominus y\leq x\ominus(x\ominus i)=x\wedge i\leq i.
\]
Since \(I\) is downward closed and \(i\in I\), we get \(x\ominus y\in I\).

Also, since \(x\ominus i\leq x\), we have by (6) in Lemma \ref {Lem:1}
\[
        y=(x\ominus i)\oplus j\leq x\oplus j.
\]
Hence, by (6) in Lemma \ref {Lem:1},
\[
        y\odot x^\ast\leq (x\oplus j)\odot x^\ast,
\]
that is, by (4) in Lemma \ref {Lem:1}
\[
        y\ominus x\leq (x\oplus j)\ominus x\leq j.
\]

Since \(I\) is downward closed and \(j\in I\), it follows that \(y\ominus x\in I\). Thus \(y\in x/I\). The equality follows.
\end{proof}

Recall that a space \(X\) is called \textbf{compact} if for any open cover \(\mathcal {U}\) there is a subcover \(\mathcal {V}\subseteq \mathcal {U}\) which has only finite elements. It is well known that the product of any finite number of compact spaces is compact, and this result does not require any separation axioms \cite{Engelking1989}. Let \(f:X\rightarrow Y\) a continuous map. Then \(f\) is called a \textbf{ perfect map} if \(f\) is closed and \(f^{-1}(y)\) is compact in \(X\) for each \(y\in Y\).

\begin{theorem}\label{thm:compact-ideal-closed-quotient}
Let \(A\) be a topological \(MV\)-algebra and let \(I\) be a compact ideal of \(A\). Let \( q:A\longrightarrow A/I
\) be the natural quotient \(MV\)-homomorphism, where \(A/I\) is endowed with the quotient topology. Then \(q\) is an open and perfect mapping.
\end{theorem}

\begin{proof}
Clearly, \(q\) is a quotient mapping in topology sense, so \(q\) is a continuous and open mapping by Corollary \ref{cor:natural-quotient-open}.

Next, we shall prove that \(q\) is a closed mapping. Let \(F\subseteq A\) be closed.  We prove that \(q(F)\) is closed in \(A/I\).  Since \(q\) is a quotient map, it is enough to show that
\[
        q^{-1}(q(F))
\]
is closed in \(A\).

Define
\[
        \Phi:A\times I\times I\longrightarrow A
\]
by
\[
        \Phi(x,i,j)=(x\ominus i)\oplus j.
\]
The map \(\Phi\) is continuous because \(A\) is a topological \(MV\)-algebra.  Since \(I\) is compact, \(I\times I\) is compact.

By Proposition \ref{lem:class-parametrization}, for every \(x\in A\),
\[
        q^{-1}(q(x))=x/I=\Phi(\{x\}\times I\times I).
\]
Therefore
\[
\begin{aligned}
        q^{-1}(q(F))
        &=\{x\in A:q^{-1}(q(x))\cap F\neq\emptyset\}\\
        &=\{x\in A:\Phi(\{x\}\times I\times I)\cap F\neq\emptyset\}.
\end{aligned}
\]
By Lemma \ref{lem:compact-parameter-closed}, this set is closed in \(A\). Hence \(q^{-1}(q(F))\) is closed.  Since \(q\) is a quotient map, \(q(F)\) is closed in \(A/I\).  Thus \(q\) is closed.

For every \(x\in A\), Proposition \ref{lem:class-parametrization} gives
\[
        q^{-1}(q(x))=x/I=\Phi_x(I\times I),
\]
where
\[
        \Phi_x(i,j)=(x\ominus i)\oplus j.
\]
Since \(I\) is compact, \(I\times I\) is compact. Hence, \(x/I\) is compact as a continuous image of \(I\times I\) because \(\Phi_x\) is continuous. Hence all fibres are compact. Therefore \(q\) is continuous,
closed, and has compact fibres; hence \(q\) is perfect.
\end{proof}

\begin{corollary}\label{cor:compact-ideal-closed-quotient}
Let \(A\) be a topological \(MV\)-algebra, let \(B\) be an \(MV\)-algebra with a topology, and let 
\( q:A\longrightarrow B\)
is a surjective \(MV\)-homomorphism which is a quotient map in the topological sense. If  \(\operatorname{Ker} q\) is compact in \(A\), then \(q\) is an open and perfect map.
\end{corollary}

\begin{proof}
Use Corollary \ref{cor:1}, identify \(B\) with \(A/ \operatorname{Ker} q\), and apply
Theorem \ref{thm:compact-ideal-closed-quotient}.
\end{proof}

\section{Applications}\label{Sec:2} \label{app}

Now we shall apply results above to study the extension in topological \(MV\)-algebras. We first derive stronger consequences under mild separation assumptions.

 Recall that a topological space is \textbf{regular} if for all \(x\in X\) and each neighborhood \(U\) of \(x\), there exists a
neighborhood \(V\) of \(x\) such that \(\overline{V}\subseteq U\), where \(\overline{V}\) is the closure of \(V\). A topological space \(X\) is \textbf{\(T_3\)} if \(X\) is regular and \(T_1\).
The following result improves \cite[ Theorem 3.6 (3)]{LY}.

\begin{proposition}\label{prop:T1-iff-closed-ideal}
Let \(A\) be a topological \(MV\)-algebra and \(I\) an ideal of \(A\). Then \(A/I\) endowed with the quotient topology is \(T_3\) if and only if \(I\) is closed in \(A\).
\end{proposition}

\begin{proof}
Let \(q:A\longrightarrow A/I\) be the natural quotient map.

Assume first that \(A/I\) is \(T_3\).  Then \(\{0/I\}\) is closed in \(A/I\), and hence
\[
        I=q^{-1}(\{0/I\})
\]
is closed in \(A\).

Conversely, suppose that \(I\) is closed.  For any \(x\in A\),
\[
        q^{-1}(\{x/I\})=x/I=\{y\in A:x\ominus y\in I\text{ and }y\ominus x\in I\}.
\]
The maps \(y\mapsto x\ominus y\) and \(y\mapsto y\ominus x\) are continuous.  Since \(I\) is closed, the set \(x/I\) is closed in \(A\). Therefore every singleton in \(A/I\) has closed inverse image under the quotient map \(q\), and hence every singleton in \(A/I\) is closed.  Thus \(A/I\) is \(T_1\). According to Corollary \ref{cor:natural-quotient-open} it follows that \(A/I\) is a topological \(MV\)-algebra. Hence, \(A/I\) is regular follows from the fact that every \(T_1\) topological \(MV\)-algebra is \(T_3\) \cite[Corollary 3.10]{GLD}.
\end{proof}

Since every compact set in Hausdorff space is closed, we can obtain the following result by Proposition \ref{prop:T1-iff-closed-ideal}.

\begin{corollary}\label{cor:Hausdorff-compact-ideal-quotient}
If \(A\) is a Hausdorff topological \(MV\)-algebra and \(I\) is a compact ideal of \(A\), then \(A/I\) is \(T_3\).
\end{corollary}

A topological space \(X\) is called \textbf{locally compact} if for each \(x\in X\) there is a compact neighborhood \(V\) of \(x\). It is well known that  a Hausdorff space \(X\) is locally compact if and only if each point in \(X\)
has a neighbourhood base of compact sets. That is, for each \(x\in X\) and each neighborhood \(V\) of \(x\) there is a neighborhood \(U\) of \(x\) such that \(\overline{U}\subseteq V\) satisfying \(\overline{U}\) being compact in \(X\).
The following result is known in folklore; for the sake of completeness, we provide a proof.

\begin{lemma}\label{lem:perfect-preimage-compact}
Let \(f:X\longrightarrow Y\) be a perfect map. If \(K\subseteq Y\) is compact, then \(f^{-1}(K)\) is compact.
\end{lemma}

\begin{proof}
Let \(\mathcal U\) be an open cover of \(f^{-1}(K)\). For each \(y\in K\), the fibre \(f^{-1}(y)\) is compact, so there are finitely many members \(U_{y,1},\ldots,U_{y,n_y}\) of \(\mathcal U\) such that
\[
        f^{-1}(y)\subseteq U_y:=U_{y,1}\cup\cdots\cup U_{y,n_y}.
\]
Since \(f\) is closed and \(X\setminus U_y\) is closed, the set
\[
        V_y=Y\setminus f(X\setminus U_y)
\]
is an open neighbourhood of \(y\). Moreover,
\[
        f^{-1}(V_y)\subseteq U_y.
\]
The family \(\{V_y:y\in K\}\) covers the compact set \(K\). Choose \(y_1,\ldots,y_m\in K\) such that
\[
        K\subseteq V_{y_1}\cup\cdots\cup V_{y_m}.
\]
Then
\[
        f^{-1}(K)\subseteq U_{y_1}\cup\cdots\cup U_{y_m},
\]
and the right hand side is covered by finitely many members of \(\mathcal U\). Hence \(f^{-1}(K)\) is compact.
\end{proof}

\begin{theorem}\label{thm:compact-kernel-three-space}
Let \(A\) be a topological \(MV\)-algebra and let \(I\) be a compact ideal of \(A\). Then for each of the following properties \(\mathcal P\), the algebra \(A\) has \(\mathcal P\) if and only if so has the quotient space \(A/I\):
\begin{enumerate}
\item[(1)] compactness;
\item[(2)] local compactness;
\item[(3)]\(\sigma\)-compactness;
\item[(4)] Lindel\"ofness.
\end{enumerate}
\end{theorem}

\begin{proof}
Let \(q:A\longrightarrow A/I\) be the natural quotient \(MV\)-homomorphism. Then the quotient map \(q\) is open and perfect by Theorem \ref{thm:compact-ideal-closed-quotient}. We verify the stated properties.

(1) If \(A\) is compact, then \(A/I\) is compact as a continuous image of \(A\).
Conversely, if \(A/I\) is compact, then \(A\)
is compact as a perfect preimage of \(A/I\) by Lemma \ref{lem:perfect-preimage-compact}.

(2) Suppose first that \(A\) is locally compact. Let \(b\in A/I\) and \(x\in A\) such that \(q(x)=b\). Take a compact neighborhood \(U\) of \(x\). Then
\(q(U)\) is compact neighborhood of \(b\) as continuous open image of \(U\) because \(q\) is continuous and open.

Conversely, suppose first that \(A/I\) is locally compact. Take an \(x\in A\). Then there is a compact neighborhood \(U\) of \(q(x)\). Hence \(q^{-1}(U)\) is a compact neighborhood of \(x\) by Lemma \ref{lem:perfect-preimage-compact}.

(3) If \(A\) is \(\sigma\)-compact, then its continuous image \(A/I\) is \(\sigma\)-compact. Conversely, if
\[
        A/I=\bigcup_{n=1}^\infty K_n
\]
with each \(K_n\) compact, then
\[
        A=\bigcup_{n=1}^\infty q^{-1}(K_n),
\]
and each \(q^{-1}(K_n)\) is compact by Lemma \ref{lem:perfect-preimage-compact}.

(4) If \(A\) is Lindel\"of, then \(A/I\) is Lindel\"of as a continuous image. Conversely, suppose \(A/I\) is Lindel\"of and let \(\mathcal U\) be an open cover of \(A\). For each \(y\in A/I\), the compact fibre \(q^{-1}(y)\) is covered by finitely many members of \(\mathcal U\); let their union be \(U_y\). As in the proof of Lemma \ref{lem:perfect-preimage-compact}, the set
\[
        V_y=(A/I)\setminus q(A\setminus U_y)
\]
is an open neighbourhood of \(y\) and satisfies \(q^{-1}(V_y)\subseteq U_y\). Since \(A/I\) is Lindel\"of, choose countably many \(y_k\) such that \(\{V_{y_k}:k\in\mathbb N\}\) covers \(A/I\). The corresponding countable union of finite subfamilies of \(\mathcal U\) covers \(A\). Thus \(A\) is Lindel\"of.

\begin{theorem}\label{5.5}
Let \(A\) be a Hausdorff topological \(MV\)-algebra and let \(I\) be a compact ideal of \(A\). Then \(A\) is paracompact if and only if so is the quotient space \(A/I\).
\end{theorem}
Let \(q:A\longrightarrow A/I\) be the natural quotient \(MV\)-homomorphism. Then the quotient map \(q\) is open and perfect by Theorem \ref{thm:compact-ideal-closed-quotient}.

Since \(A\) is a Hausdorf topological \(MV\)-algebra, \(A\) is \(T_3\) by \cite[Corollary 3.10]{GLD}. Hence \(A/I\) as a perfect image of \(A\) is paracompact by \cite[Theorem 5.1.33]{Engelking1989}.
Conversely, it follows from the fact that paracompactess is an inverse invariant of perfect maps \cite[Theorem 5.1.35]{Engelking1989}.
\end{proof}

\begin{lemma}\label{Lem:ideal}
Let \(I\) be an ideal of a topological \(MV\)-algebra such that \(I\) has a countable neighborhood base in \(A\). If \(0 \) has a countable neighborhood base in the subspace \(I\), then \(A\) is a first countable space.
\end{lemma}

\begin{proof}

According to \cite[Lemma 3.4]{GLD}, it is enough to show that the point \(0 \) has a countable neighborhood base in \(A\).

Let \(\{U_n\}_{n \in \mathbb N}\) be a countable neighborhood base of \(I\) in \(A\)  and \(\{G_m\}_{m \in \mathbb N}\) a countable neighborhood base of \(0\) in \(I\).

For each \(m\), since \(G_m\) is open in \(I\), there exists an open set \(H_m \subset A\) such that
\[
G_m = H_m \cap I.
\]

Since every topological \(MV\)-algebra is regular \cite[Theorem 3.9]{GLD}, one can take an open neighborhood \(W_m\) of \(0\) in \(A\) such that the closure \(\operatorname{cl}(W_m)\) of \(W_m\) satisfying
\[
\operatorname{cl}(W_m) \cap I \subset G_m \tag{1}
\]
for each \(m\in \mathbb N\).

We claim that the countable family \[\mathcal B=\{U_i\cap W_j:i,j\in \mathbb N\} \] is a neighborhood base at \(0\) in \(A\).

In fact, let \(O\) be any open neighborhood of \(0\) in \(A\). Since \(O \cap I\) is an open neighborhood of \(0\) in the subspace \(I\), and \(\{G_m\}_{m \in \mathbb N}\) is a neighborhood base of \(0\) in \(I\), there exists some \(m\in \mathbb N\) such that
\[
G_m \subset O \cap I.
\]
Combining this with (1), we get
\[
\operatorname{cl}(W_m) \cap I \subset G_m \subset O.
\]
Thus, we have
\[
(I\setminus O) \cap \operatorname{cl}(W_m) = \varnothing.
\]
Hence, we have  \(A\setminus \operatorname{cl}(W_m)\) is open in \(A\) satisfying
\[
(I\setminus O)\subseteq A\setminus \operatorname{cl}(W_m).
\]
Furthermore,
\[
I= (I \cap O) \cup (I \setminus O) \subset O \cup (A\setminus \operatorname{cl}(W_m)).
\]
Since \(\{U_n\}_{n \in \mathbb N}\) is a countable neighborhood base of \(I\) in \(A\),  there exists some \(n\in\mathbb N\) such that
\[
I\subseteq U_n \subset O \cup (A\setminus \operatorname{cl}(W_m)).
\]
Hence, we obtain
\[
U_n \cap W_m \subset  (O \cup (A\setminus \operatorname{cl}(W_m))) \cap W_m = O \cap W_m \subset O.
\]
 Hence \(\mathcal B\) is a countable neighborhood base of \(x\) in \(A\). This proves first countability of \(A\).
\end{proof}

\begin{theorem}\label{Them:5.7}
Let \(A\) be a topological \(MV\)-algebra and \(I\) be a compact ideal in \(A\). Then \(A\) is first countable if and only if both \(I\) and the quotient space \(A/I\) are first countable.
\end{theorem}

\begin{proof}
Let \(\pi:A\longrightarrow A/I\) be the natural quotient map. Then \(\pi\) is an open and perfect map by Theorem \ref{thm:compact-ideal-closed-quotient}.

Suppose that \(A\) is first countable. Then both \(I\) as a subspace of \(A\)  and \(A/I\) as a continuous and open image of \(A\) are first countable.

Conversely, suppose that both \(I\) and the quotient space \(A/I\) are first countable. According to Lemma \ref{Lem:ideal}, it is enough to show that \(I\) has a countable neighborhood base in \(A\).
In fact, since \(A/I\) is first countable, there is \(\{U_n\}_{n \in \mathbb N}\) is a countable open neighborhood base of \(\pi(0)\) in \(A/I\), where \(0\in A\). Then one can easily show that the countable
family \[\{\pi^{-1}(U_n)\}_{n \in \mathbb N}\] a countable neighborhood base of \(I\) in \(A\).

In fact, take an open neighborhood \(O\) of \(I\). Then \(A/I \setminus\pi(A\setminus O)\) is an open neighborhood of \(\pi(0)\) because \(\pi\) is closed. Hence, there is \(n\in \mathbb N\) such that
\[
\pi(0)\in U_n\subseteq A/I \setminus\pi(A\setminus O),\tag{1}
\]
because \(\{U_n\}_{n \in \mathbb N}\) is a countable open neighborhood base of \(\pi(0)\) in \(A/I\).
From (1) it follows that
\[I=\pi^{-1}(\pi(0))\subseteq \pi^{-1}(U_n)\subseteq \pi^{-1}( A/I \setminus\pi(A\setminus O))\subseteq O.\]
This implies that \(\{\pi^{-1}(U_n)\}_{n \in \mathbb N}\) is a countable neighborhood base of \(I\) in \(A\).
\end{proof}

\begin{lemma} \cite[4.17]{Gran} \label{para-p}
A Tychonoff space \(X\) admits a perfect mapping onto a metrizable
space if and only if \(X\) is a paracompact \(p\)-space.
\end{lemma}

Let \(X\) be a topological space and \(\Delta=\{(x,x)\in X\times X:x\in X\}\). If \(\Delta\) is a \(G_\delta\)-set in the space \(X\times X\), that is, \(\Delta\) can be written as a countable intersection of open sets in \(X\times X\), then we call \(X\) with a \textbf{\(G_\delta\)-diagonal}.

\begin{lemma} \label{Lem:dia}
Every \(T_1\) first countable topological \(MV\)-algebra \(A\) has a \( G_\delta\)-diagonal.
\end{lemma}

\begin{proof}
Since \( A \) is first-countable, take a countable neighborhood base at \( 0 \):
\[
\{U_n : n \in \mathbb{N}\}.
\]
Since \( A \) is \( T_1 \), we have
\[
\bigcap_{n \in \mathbb{N}} U_n = \{0\}.
\]
Define the algebraic distance \(d_A:A\times A\longrightarrow A\) as following:
\[
d_A(x, y) = (x \ominus y) \oplus (y \ominus x).
\]
Since \(A\) is a topological \(MV\)-algebra, \(d_A\) is continuous and
\[
d_A(x, y) = 0 \iff x = y.
\]
Now set
\[
G_n = \{(x, y) \in A \times A : d_A(x, y) \in U_n\}.
\]
Each \( G_n \) is an open set in \( A \times A \) and contains the diagonal \( \Delta=\{(x,x)\in A\times A:x\in A\} \). Moreover,
\[
\bigcap_{n \in \mathbb{N}} G_n = \Delta.
\]
Therefore, \( \Delta \) is a \( G_\delta \)-set.
\end{proof}

\begin{theorem}\label{Them:5.10}
Let \(A\) be a Hausdorff topological \(MV\)-algebra and \(I\) be a compact ideal in \(A\). Then \(A\) is metrizable if and only if both \(I\) and the quotient space \(A/I\) are metrizable.
\end{theorem}

\begin{proof}
Let \(\pi:A\longrightarrow A/I\) be the natural quotient map. Then \(\pi\) is an open and perfect map by Theorem \ref{thm:compact-ideal-closed-quotient}.

Suppose that \(A\) is metrizable. Then both \(I\) as a subspace of \(A\)  and \(A/I\) as a perfect image of \(A\) are metrizable \cite[Theorem 2.2.1]{Lin}.

Conversely, suppose that both \(I\) and the quotient space \(A/I\) are metrizable. Then according to Theorem \ref{5.5} it follows that \(A\) is Hausdorf paracompact. Hence, \(A\) is \(T_3\) paracompact space by the fact that every \(T_1\) topological \(MV\)-algebra is \(T_3\) \cite[Corollary 3.10]{GLD}. Thus, \(A\) is Tychonoff. Since \(\pi\) is a perfect map and \(A/I\) is metrizable, \(A\) is a paracompact \(p\)-space by Lemma \ref{para-p}. By Theorem \ref{Them:5.7}, \(A\) is first countable, so \(A\) has a \(G_\delta\)-diagonal by Lemma \ref{Lem:dia}. Hence, \(A\) is metrizable because every regular paracompact \(p\)-space with a \(G_\delta\)-diagonal is metrizable \cite[Theorem 2.2.12 and Corollary 2.2.19]{Lin}.
\end{proof}

\begin{corollary}\label{cor3:compact-kernel-three-space}
Let \(A\) be a Hausdorff topological \(MV\)-algebra and \(I\) be a compact ideal in \(A\). Then \(A\) is separable and metrizable if and only if both \(I\) and the quotient space \(A/I\) are separable and metrizable.
\end{corollary}

\begin{proof}
Let \(\pi:A\longrightarrow A/I\) be the natural quotient map. Then \(\pi\) is an open and perfect map by Theorem \ref{thm:compact-ideal-closed-quotient}.

Firstly, we shall prove that \(A\) is separable if both \(I\) and the quotient space \(A/I\) are separable.

Let \(M_1\) be a countable dense subset of \(I\) and \(M_2\) a countable subset of \(A\) such that \(\pi(M_2)\) is dense in \(A/I\). We have the following Claim 1:

{\bf Claim 1:} \(M=\{(i\ominus j)\oplus k:i\in M_2, j,k\in M_1\}\) is dense in \(A\).

In fact, take any nonempty open set \(O\) in \(A\). Since \(\pi\) is open and \(\pi(M_2)\) is dense in \(A/I\), there is \(i\in M_2\) such that \(\pi(i)\in \pi(O)\), which implies that
 \[i/I\cap O\neq\emptyset \tag{1}\]
Since \(M_1\) is dense in \(I\), one can easily verify \(M_1\times M_1\) is dense in \(I\times I\). Hence,
\[\Phi_i(M_1\times M_1)=\{(i\ominus j)\oplus k: j,k\in M_1\}\subseteq \Phi_i(I\times I)=i/I \]
is dense in \(i/I\) because \(\Phi_i\) is continuous, where the last equality follows from Proposition \ref{lem:class-parametrization} and \(\Phi_i\) is defined in
Theorem \ref{thm:compact-ideal-closed-quotient}. Hence, form (1) there are \(j,k\in M_1\) such that
 \[(i\ominus j)\oplus k \in O,\]
which implies that \(M\) is dense in \(A\) because \((i\ominus j)\oplus k \in M\). This completes the proof of Claim 1.

From Claim 1 and Theorem \ref{Them:5.10} it follows that \(A\) is separable and metrizable if both \(I\) and the quotient space \(A/I\) are separable and metrizable.

Suppose that \(A\) is separable and metrizable. Then \(I\) as a subspace of \(A\) is separable and metrizable. Also, \(A/I\) as a continuous image of \(A\) is separable.
Applying Theorem \ref{Them:5.10} again, it follows that \(A/I\) is metrizable.
\end{proof}

\begin{corollary}\label{cor3:compact-kernel-three-space}
Let \(A\) be a Hausdorff topological \(MV\)-algebra and \(I\) be a compact ideal in \(A\). Then \(A\) is compact and metrizable if and only if both \(I\) and the quotient space \(A/I\) are compact and metrizable.
\end{corollary}

\begin{proof}
This directly follows from Theorems \ref{thm:compact-kernel-three-space} and \ref{Them:5.10}.
\end{proof}

A subset \(C\) of an \(MV\)-algebra \(A\) is called \textbf{convex} if, for all \(x, y \in C\) and \(z \in A\), the condition
\(x \leq z \leq y\) implies \(z \in C\). A topological \(MV\)-algebra \(A\) is called \textbf{local convex} if every point \(a \in A\)
possesses a convex neighborhood base.
\begin{theorem}\label{cor3:compact-kernel-three-space}
Let \(A\) be a local convex topological \(MV\)-algebra and let \(I\) be an ideal of \(A\). Then \(A\) is first countable if and only if both \(I\) and \(A/I\) are first countable.
\end{theorem}

\begin{proof}
Let \(\pi:A\longrightarrow A/I\) be the natural quotient map. Then \(\pi\) is an open and continuous map by Corollary \ref{cor:natural-quotient-open}.

Suppose that \(A\) is first countable. Then both \(I\) as a subspace of \(A\)  and \(A/I\) as a continuous and open image of \(A\) are first countable.

Conversely, suppose that both \(I\) and \(A/I\) are first countable. Since \(A\) is local convex topological \(MV\)-algebra, one can construct a countable family \(\{U_n : n\in\mathbb{N} \}\) of
open neighborhoods of \(0\) satisfying the following conditions for each \(n\in\mathbb{N}\):
\begin{enumerate}
\item [(1)] \( U_{n+1}\subseteq U_n;\)
\item [(2)] \(\{U_n\cap I: n\in\mathbb{N}\}\) is a neighborhood base at \(0\) in \(I\);
\item [(3)] each \(U_n\) is convex.
\end{enumerate}
Since \(\pi\) is continuous and open, and \(A/I\) is first countable, we can choose a countable neighborhood base \(\{V_n:n\in\mathbb{N}\}\) at \(\pi(0)\) in \(A/I\) such that \(V_n\subseteq \pi(U_n)\) and \(V_{n+1}\subseteq V_n\) for each \(n\in\mathbb{N}\).
We claim that the family
\[\mathcal {B}=\{U_n\cap \pi^{-1}(V_n):n\in \mathbb{N}\}\]
is a local base at \(0\) in \(A\).

Take an open neighborhood \(U\) of \(0\). Since \(A\) is a local convex topological \(MV\)-algebra, one can choose an open and convex neighborhood \(W\) of \(0\) such that
 \[W\oplus W\subseteq U. \tag{a}\]
By (2), there is an \(m\in \mathbb{N}\) such that
\[U_m\cap I\subseteq W\cap I.\tag{b}\] Hence, there is \(V_n\) such that \(V_n\subseteq \pi(W)\) because \(\pi\) is open and  \(\{V_i:i\in\mathbb{N}\}\) is a neighborhood base at \(\pi(0)\).
Put \(k=\text{max~}\{m,n\}\). Then \(U_k\subseteq U_m\) by (1). Hence form (b) we have
\[U_k\cap I\subseteq W\cap I.\tag{c}
\]
We claim that \[U_k\cap \pi^{-1}(V_k)\subseteq U.\]

Take any \(x\in U_k\cap \pi^{-1}(V_k)\). Then
\[\pi(x)\in V_k\subseteq V_n\subseteq \pi(W),\]
hence, there is \(y\in W\) such that \(\pi(x)=\pi(y)\).
This implies that \(x\ominus y\in I\). Also \(x\ominus y\leq x\) and \(x\in U_k\). Because \(U_k\) is convex and contains \(0\), it is downward closed; hence \(x\ominus y \in U_k\). Therefore, by (c)
\[x\ominus y \in U_k\cap I\subseteq W \tag{d}
\]
It is obvious that \(x\wedge y\leq y\). Hence, from the fact that \(y\in W\) and \(W\) is convex,  it follows
\[x\wedge y\in W. \tag{e}\]
Combining (a), (d),(e) with (1) in Lemma \ref{Lem:1}, we have
\[x=((x\wedge y)\oplus (x\ominus y)\in W\oplus W\subseteq U.\]
This implies that \(U_k\cap \pi^{-1}(V_k)\subseteq U\). Thus we have proved that \(\mathcal {B}\) is a countable neighborhood base at \(0\) in \(A\). Then \(A\) is first countable by \cite[Lemma 3.4]{GLD}.
\end{proof}
\begin{question}

Let \(A\) be a topological \(MV\)-algebra and let \(I\) be an ideal of \(A\). Is \(A\) first countable if both \(I\) and \(A/I\) are first countable?
\end{question}

\section*{References}


\def\cprime{$'$}

\end{document}